\documentclass{amsart}
\usepackage{amssymb,stmaryrd}
\usepackage{enumerate,enumitem}
\usepackage[colorlinks,citecolor=blue,urlcolor=black,linkcolor=black]{hyperref}

\newtheorem{thm}{Theorem}[section]

\newtheorem{cor}[thm]{Corollary}
\newtheorem{claim}{Claim}[thm]
\newtheorem{mainthm}{Theorem}

\theoremstyle{definition}
\newtheorem{defn}[thm]{Definition}

\DeclareMathOperator{\cf}{cf}
\DeclareMathOperator{\ns}{NS}
\DeclareMathOperator{\im}{Im}
\DeclareMathOperator{\ssup}{ssup}
\DeclareMathOperator{\nacc}{nacc}
\DeclareMathOperator{\acc}{acc}
\DeclareMathOperator{\dom}{dom}
\DeclareMathOperator{\otp}{otp}
\DeclareMathOperator{\Top}{top}

\newcommand{\s}{\subseteq}  
\newcommand{\sq}{\sqsubseteq}  
\newcommand{\br}{\blacktriangleright}
\newcommand*\axiomfont[1]{\textsf{\textup{#1}}}
\newcommand\zfc{\axiomfont{ZFC}}
\newcommand\ch{\axiomfont{CH}}
\renewcommand\top[1]{{\Top\,#1}}
\renewcommand\mid{\mathrel{|}\allowbreak}
\renewcommand{\restriction}{\mathbin\upharpoonright}    

\title[A counterexample related to Komj\'ath and Weiss]{A counterexample related to\\a theorem of Komj\'ath and Weiss}

\author {Rodrigo Carvalho}
\address{Department of Mathematics, Bar-Ilan University, Ramat-Gan 5290002, Israel.}

\author {Assaf Rinot}
\address{Department of Mathematics, Bar-Ilan University, Ramat-Gan 5290002, Israel.}
\urladdr{http://www.assafrinot.com}

\begin{document}
\begin{abstract} In a paper from 1987, Komj\'ath and Weiss proved that for every regular topological space $X$
of character less than $\mathfrak b$, if $X\rightarrow(\top{\omega+1})^1_\omega$, then $X\rightarrow(\top \alpha)^1_\omega$ for all $\alpha<\omega_1$.
In addition, assuming $\diamondsuit$, they constructed a space $X$ of size continuum, of character $\mathfrak b$,
satisfying $X\rightarrow(\top{\omega+1})^1_\omega$, but not $X\rightarrow(\top{\omega^2+1})^1_\omega$.
Here, a counterexample space with the same characteristics is obtained outright in $\zfc$.
\end{abstract}
\dedicatory{This paper is dedicated to Istv\'{a}n Juh\'{a}sz on the occasion of his 80th birthday}
\date{A preprint as of July 2, 2023 (Happy Birthday Istv\'{a}n!). For the latest version, visit \textsf{http://www.assafrinot.com/paper/61}.}
\maketitle

\section{Introduction}

For two topological spaces $X,Y$ and a cardinal $\theta$,  the arrow notation
$$X\rightarrow(\top Y)^1_\theta$$
asserts that for every coloring $c:X\rightarrow\theta$,
there exists an homeomorphism $\phi$ from $Y$ to $X$ such that $c$ is constant over $\im(\phi)$.

In \cite{MR911048}, Komj\'{a}th and Weiss studied the partition relation $X\rightarrow(\top\alpha)^1_\omega$, where $\alpha$ is a countable ordinal endowed with the usual order topology.
The first result of their paper is a pump-up theorem
for regular topological spaces of character less than $\mathfrak b$;\footnote{Recall that $\mathfrak b$ denotes the least size of an unbounded submfaily of ${}^\omega\omega$,
where a subfamily $\mathcal F\s{}^\omega\omega$ is \emph{bounded} if, for some function $g:\omega\rightarrow\omega$,
$\{ n<\omega\mid f(n)\le g(n)\}$ is finite for all $f\in\mathcal F$.}
the theorem asserts that for any such space $X$, if $X\rightarrow(\top{\omega+1})^1_\omega$, then moreover $X\rightarrow(\top \alpha)^1_\omega$ for all $\alpha<\omega_1$.\footnote{The published proof had a small gap that was later rectified in \cite{MR4550370} based on a suggestion of Weiss.}

To show that the bound $\mathfrak b$ cannot be improved, Theorem~4 of \cite{MR911048} 
gives an example, assuming $\diamondsuit$, of a regular topological space $X$ of size and character $\aleph_1$
such that $X\rightarrow(\top{\omega+1})^1_\omega$, but not $X\rightarrow(\top{\omega^2+1})^1_\omega$.
Question~2 of the same paper asks whether there is a $\zfc$ example of a regular space $X$
satisfying $X\rightarrow(\top{\omega+1})^1_\omega$ and failing $X\rightarrow(\top{\alpha})^1_\omega$ for some countable ordinal $\alpha>\omega^2$.
The first main result of this note answers this question in the affirmative.

\begin{mainthm}\label{thma} There exists a zero-dimensional regular space $X$ of size continuum, of character $\mathfrak b$,
satisfying $X\rightarrow(\top{\omega+1})^1_\omega$, but not $X\rightarrow(\top{\omega^2+1})^1_1$.
\end{mainthm}

In \cite{MR4550370}, the Komj\'{a}th-Weiss counterexample was addressed from a different angle.
There, a weakening of $\diamondsuit$ called $\clubsuit_F$ was introduced and shown to be sufficient for the construction of the same $\aleph_1$-sized example.
Furthermore, it is established there that $\clubsuit_F$ is consistent with the failure of $\ch$.
Here, we provide an alternative way to get an $\aleph_1$-sized countexample space together with a large continuum:

\begin{mainthm}\label{thmb} After forcing to add any number of Cohen reals, there exists a zero-dimensional regular space $X$ of size $\aleph_1$, of character $\mathfrak b$,
satisfying $X\rightarrow(\top{\omega+1})^1_\omega$, but not $X\rightarrow(\top{\omega^2+1})^1_1$.
\end{mainthm}

The preceding is a special case of a general theorem that identifies a class of notions of forcing that inevitably add consequences of higher analogs of $\clubsuit_F$.
These notions of forcing include Cohen forcing, but also Prikry and Magidor forcing.

\subsection{Notation and conventions}
For a regular cardinal $\kappa$, we denote by $H_\kappa$
the collection of all sets of hereditary cardinality less than $\kappa$.
$E^\kappa_\chi$ denotes the set $\{\alpha < \kappa \mid \cf(\alpha) = \chi\}$, and
$E^\kappa_{\geq \chi}$, $E^\kappa_{<\chi}$, $E^\kappa_{\neq\chi}$, etc.\ are defined analogously.

For a set of ordinals $a$, we write $\ssup(a) := \sup\{\alpha + 1 \mid\alpha \in a\}$, $\acc^+(a) := \{\alpha < \ssup(a) \mid \sup(a \cap \alpha) = \alpha > 0\}$,
$\acc(a) := a \cap \acc^+(a)$, and $\nacc(a) := a \setminus \acc(a)$.

\section{Topological spaces based on trees}\label{sec1}

Following \cite{paper23}, we say that $T$ is a \emph{streamlined tree} iff there exists some cardinal $\kappa$ such that $T\s{}^{<\kappa}H_\kappa$
and, for all $t\in T$ and $\alpha<\dom(t)$, $t\restriction\alpha\in T$. 
For a subset $E\s\kappa$, we let $T\restriction E:=\{t\in T\mid \dom(t)\in E\}$.
For a subset $T'\s T$, a \emph{ladder system} over $T'$ is a sequence $\vec A=\langle A_t\mid t\in T'\rangle$ such that,
for every  $t\in T'$, $A_t$ is a cofinal subset of $t_\downarrow:=\{ s\in T\mid s\subsetneq t\}$ with $\otp(A_t)=\cf(\dom(t))$.
For every ladder system $\vec A=\langle A_t\mid t\in T'\rangle$, we attach a symmetric relation $E_{\vec A}\s[T]^2$, as follows:
$$E_{\vec A} = \{ \{s,t\} \mid t \in T',s \in A_t \}.$$

\begin{thm}\label{spacefromtree} Suppose that:
\begin{itemize}
\item $T\s{}^{<\kappa}H_\kappa$ is a streamlined tree;
\item $\vec A$ is a ladder system over $T'=T\restriction E^\kappa_\omega$;
\item The graph $(T,E_{\vec A})$ is uncountably chromatic.
\end{itemize}

Then there exists a zero-dimensional topology $\tau$ on $T$ such that $X:=(T,\tau)$ is a regular space of character $\mathfrak{b}$ satisfying $X \rightarrow (\top(\omega + 1))^1_{\omega}$ and $X \nrightarrow (\top(\omega^{2} + 1))^1_1$.
\end{thm}

\begin{proof}
Let $\vec{A}=\langle A_t\mid t\in T'\rangle$ denote the above ladder system.
We now build another ladder system $\langle B_t\mid t\in T'\rangle$ with the property that $A_t\cap T'\s B_t\s T\restriction(E^\kappa_1\cup E^\kappa_{\omega})$ for all $t\in T'$.
To this end, for each $t\in T'$, we consider three options:
\begin{itemize}
\item[$\br$] If $A_t\cap T'$ is infinite, then let $B_t:=A_t\cap T'$. 

\item[$\br$] If $A_t\cap T'$ is finite, but $t\in T\restriction(E^\kappa_\omega\cap\acc^+(E^\kappa_\omega))$, 
then let $B_t$ be some cofinal subset of $t_\downarrow\cap E^\kappa_\omega$ of order-type $\omega$, with $A_t\cap T'\s B_t$.

\item[$\br$] Otherwise, let $B_t$ be some cofinal subset of $t_\downarrow$ of order-type $\omega$ all of whose nodes $s$ with $\cf(\dom(s))\neq 1$ are the ones from $A_t\cap T'$.
\end{itemize}

Next, for every $t\in T$ and every $i<\omega$, we consider two cases depending on whether $B_t(i)$ --- the $i^{\text{th}}$-element of $B_t$ --- belongs to $T'$:
\begin{itemize}
\item[$\br$] If $B_t(i)\in T'$, then let $\langle a_{t,i}(j)\mid j <\omega\rangle$ be a strictly increasing sequence of nodes converging to $B_{t}(i)$.
We also require that $a_{t,i+1}(0)$ be bigger than $B_t(i)$ for all $i<\omega$.
\item[$\br$] Otherwise, let $\langle a_{t,i}(j)\mid j <\omega\rangle$  be the constant sequence whose sole element is $B_{t}(i)\restriction(\max(\dom(B_t(i))))$.
\end{itemize}

$$\frame{
\setlength{\unitlength}{0.8cm}
\begin{picture}(18.3,2)
\put(1.8,0.25){\textsf{Figure 1: Illustration of the ladders assigned to a node $t\in T\restriction E^\kappa_\omega\cap\acc^+(E^\kappa_\omega)$.}}
\put(0,1){\vector(1,0){17.5}}
\linethickness{0.6mm}
\put(17.7,0.9){$t$}

\put(0.65,0.92){$|$}
\put(0.3,1.45){$a_{t,0}(0)$}
\put(1.7,1.35){$\cdots$}
\put(2.8,0.92){$|$}
\put(2.5,1.45){$a_{t,0}(j)$}
\put(4.1,1.35){$\cdots$}
\put(5.45,0.75){\line(0,0){0.5}}
\put(5,1.45){$B_t(0)$}
\put(6.5,0.92){$|$}
\put(7,0.92){$|$}
\put(7.5,0.75){\line(0,0){0.5}}
\put(8,0.92){$|$}
\put(8.5,0.92){$|$}
\put(8.2,1.35){$\cdots$}
\put(7.6,1.35){$\cdots$}
\put(7,1.35){$\cdots$}
\put(6.4,1.35){$\cdots$}
\put(9.45,0.75){\line(0,0){0.5}}
\put(9.05,1.45){$B_t(i)$}
\put(11.5,0.92){$|$}
\put(12,0.92){$|$}
\put(12.5,0.92){$|$}
\put(11,0.92){$|$}
\put(10.4,1.45){$a_{t,i+1}(0)$}
\put(13.3,0.92){$|$}
\put(12.6,1.45){$a_{t,i+1}(j)$}
\put(14,0.92){$|$}
\put(14.5,0.92){$|$}
\put(15.4,0.75){\line(0,0){0.5}}

\put(14.6,1.4){$B_{t}(i+1)$}
\put(16.7,1.25){$\cdots$}
\end{picture}
}$$		

\begin{claim}\label{claim211} There exist a family $\mathcal F\s{}^\omega\omega$ of size $\mathfrak b$ such that:
\begin{itemize}
\item for every $A\in[\omega]^\omega$, for every function $g:A\rightarrow\omega$, there exists $f\in\mathcal F$ for which $\{ n\in A\mid g(n)\le f(n)\}$ is infinite;
\item $\mathcal F$ is closed under pointwise maximum, i.e., for all $f,g\in\mathcal F$, the function $n\mapsto\max\{f(n),g(n)\}$ is in $\mathcal F$, as well.
\end{itemize}
\end{claim}
\begin{proof} This is well-known, but we include an argument anyway.
By \cite[Proposition~2.4]{paper51}, $m_f(\omega,\omega,\omega,\omega)=\mathfrak b$, hence, we may fix a family $\mathcal H$ of functions from $\omega$ to $[\omega]^{<\omega}$
such that, for every $A\in[\omega]^\omega$, and every function $g:A\rightarrow\omega$, there exists $h\in\mathcal H$ for which $\{ n\in A\mid g(n)\in h(n)\}$ is infinite.
Now, let $\mathcal F$ denote the smallest subfamily of ${}^\omega\omega$ that covers $\{\sup\mathrel{\circ} h\mid h \in \mathcal{H}\}$
and that is closed under pointwise maximum.\footnote{We use $\sup$ instead of $\max$, since $\sup(x)$ is meaningful for any set $x$, including $x=\emptyset$.}
Clearly, $\mathcal F$ is as sought.
\end{proof}

Let $\mathcal F$ be given by the claim. For all $s,t\in T$, denote $(s,t]:=\{ x\in T\mid s\s x\subsetneq t\}$.
We shall now define a topology $\tau$ over $T$ by defining a system $\langle \mathcal N_t\mid t\in T\rangle$ of local bases.
For every $t \in T \setminus T'$, set $\mathcal N_t:=\{\{t\}\}$.
For every $t \in T'$, set $\mathcal N_t:=\{ N_t(f,j)\mid f\in\mathcal F, j<\omega\}$, where
$$N_{t}(f,j) =\{t\} \cup\biguplus\{(a_{t,i}(f(i)), B_{t}(i)] \mid j\le i <\omega\}.$$

Since $\mathcal F$ is closed under pointwise maximum, $\mathcal N_t$ is indeed closed under intersections. 
In addition, for every element $s$ of a neighborhood $N_{t}(f,j)$, there exists $N\in\mathcal N_s$ with $N\s N_t(f,j)$. Specifically:
\begin{itemize}
\item If $s\in T\setminus T'$, then $N:=\{s\}$ does the job;
\item If $s\in T'\setminus\{t\}$, then there exists a unique $i\in\omega\setminus j$ such that $s\in (a_{t,i}(f(i)), B_{t}(i)]$,
and so by picking a large enough $k$ to satisfy $(a_{t,i}(f(i))\s B_{s}(k)$,
we get that $N_s(g,k+1)\s N_t(f,j)$ for any choice of $g\in\mathcal F$.
\end{itemize}

As $\bigcap\mathcal N_t=\{t\}$ for every $t\in T$, we altogether conclude that $X=(T,\tau)$ is a $\mathsf T_1$ topological space.
As $|\mathcal N_t|\le|\mathcal F\times\omega|=\mathfrak b$ for every $t\in T$,
we get that $\chi(X)\le\mathfrak b$.
Since $X$ is $\mathsf T_1$, to show that $X$ is regular, it suffices to prove that the space $X$ is zero-dimensional.

\begin{claim} Every $N\in\bigcup_{t\in T}\mathcal N_t$ is $\tau$-closed.
\end{claim}
\begin{proof} Let $t\in T'$, $f\in\mathcal F$, $j<\omega$, and we shall show that that $N_t(f,j)$ is $\tau$-closed.
To this end, let $s\in T\setminus N_t(f,j)$. 

$\br$ If $s\notin T'$, then $\{s\}$ is a neighborhood of $s$ disjoint from $N_t(f,j)$.

$\br$ If $s\in T'$ and $s\s B_t(0)$, then $N_s(g,0)$ is readily disjoint from $N_t(f,j)$ for any choice of $g\in\mathcal F$.

$\br$ If $s\in T'$ and $B_t(i)\s s\s B_t(i+1)$, then find a large enough $k<\omega$ such that $B_t(i)\s B_s(k)$,
and note that $N_s(g,k+1)$ is disjoint from $N_t(f,j)$ for any choice of $g\in\mathcal F$.

$\br$ If $s\in T'$ and $s\notin t_\downarrow$, then $r:=s\cap t$ is an element of $T$ that
constitutes the meet of $s$ and $t$.
Find a large enough $k$ such that $r\s B_s(k)$
and note that for any choice of $g\in\mathcal F$, $N_s(g,k+1)$ is disjoint from $t_\downarrow$,
and hence from $N_t(f,j)$.
\end{proof}

\begin{claim}\label{claim213} $X \rightarrow (\top(\omega + 1))^1_{\omega}$.
\end{claim}
\begin{proof}
Let $c: T \rightarrow \omega$ be a given a coloring.
It suffices to find a $t \in T'$ such that $\{ s\in B_T\mid c(s)=c(t)\}$ is infinite.
Towards a contradiction, suppose that $\{ s\in B_T\mid c(s)=c(t)\}$ is finite for every $t \in T'$.
It follows that we may define a function $d: T \rightarrow \omega\times 2\times\omega$ by recursion on the levels of $T$, as follows:
$$d(t) := \begin{cases}\langle c(t), 1,\max\{0,n+1\mid \exists s\in B_t\,[c(s)=c(t)\ \&\ d(s)=\langle c(s),1,n\rangle]\}\rangle,&\text{if }t\in T'\\
\langle c(t),0,0\rangle,&\text{otherwise}.\end{cases}$$

Recalling that $(T,E_{\vec A})$ is uncountably chromatic, 
we may now find $\{s,t\}\in E_{\vec A}$ such that $d(s)=d(t)$.
By possibly switching the roles of $s$ and $t$, we may assume that $t \in T'$ and $s \in A_t$.
As $t\in T'$, it follows that $d(t)=(c(t),1,m)$ for some $m<\omega$.
As $d(s)=d(t)$, it follows that $c(s)=c(t)$ and $s\in T'$, and hence $s\in B_t$. But then the definition of $d(t)$ implies that the third coordinate of $d(t)$ is bigger than the corresponding one of $d(s)$. 
This is a contradiction.
\end{proof}

\begin{claim} $X \nrightarrow (\top(\omega^2 + 1))^1_1$.
\end{claim}
\begin{proof} Towards a contradiction, suppose that $\phi:\omega^2+1\rightarrow X$ is an homeomorphism.
For every $n<\omega$, since $\omega\cdot(n+1)$ is an accumulation point of the interval $A_n:=({\omega\cdot n},{\omega\cdot(n+1)})$,
the singleton	 $\{\phi(\omega\cdot(n+1))\}$ cannot be $\tau$-open, so that the node $t_n:=\phi({\omega\cdot(n+1)})$ must be in $T'$
and	$\phi[A_n]$ must contain an infinite sequence converging to $t_n$.
Likewise, $\{t_n\mid n<\omega\}$ must contain an infinite sequence converging to the node $t_\omega:=\phi(\omega^2)$.
It thus follows that there exists a strictly increasing and continuous map $\psi:\omega^2+1\rightarrow\omega^2+1$ such that 
$\phi\circ \psi$ is a strictly increasing and continuous map from $\omega^2+1$ to $T$.
For notational simplicity, we assume $\psi$ is the identity, so that $\langle t_n\mid n<\omega\rangle$
is a strictly increasing sequence of nodes in $T'$ converging to $t_\omega$.
In particular, $t_\omega \in T\restriction({E^{\kappa}_{\omega}\cap \acc^+(E^{\kappa}_{\omega})})$.

As $\otp(B_{t_\omega})=\omega<\omega^2=\otp(\phi[\omega^2])$, we may fix a map $d:\omega\rightarrow \phi[\omega^2]\setminus B_t$
such that $\langle d(n)\mid n<\omega\rangle$ is a strictly increasing increasing sequence of nodes converging to $t_\omega$.
Consequently, the following set is infinite:
$$A := \{i\in\omega\setminus\{0\}\mid (B_t(i-1),B_t(i)]\text{ has an element of }\im(d)\}.$$ 

It follows that for every $i\in A$, we may let $$m_i := \max\{m <\omega \mid B_t(i-1)\subsetneq d(m) \s B_t(i)\}.$$

Define a function $g:A \rightarrow \omega$ defined via $$g(i) := \min\{j <\omega \mid d(m_{n})\s a_{t,i}(j) \}.$$

Recalling that $\mathcal F$ was given by Claim~\ref{claim211},
we now pick $f \in\mathcal F$ such that $I:=\{ n\in A\mid g(n)\le f(n)\}$ is infinite.
For every $i\in I$,  it is the case that
$$B_t(i-1)\subsetneq d(m_i)\s a_{t,i}(g(i))\s a_{t,i}(f(i))\subsetneq B_t(i).$$

Therefore, for every node $s$ in the set $D:=\{ d(m_i)\mid i\in I\}$,
there exists an $i\in I$ such that $D\cap (B_t(i-1),B_t(i)]=\{ s\}$.
So $D$ is an infinite discrete subset of the compact set $\phi[\omega^2+1]$. This is a contradiction.
\end{proof}
It now follows from \cite[Theorem~1]{MR911048} that $\chi(X)\ge\mathfrak b$. Altogether, the space $X$ is as sought.
\end{proof}

We are now ready to prove Theorem~\ref{thma}.

\begin{cor} There exists a zero-dimensional regular space $X$ of size continuum, of character $\mathfrak b$,
satisfying $X\rightarrow(\top{\omega+1})^1_\omega$, but not $X\rightarrow(\top{\omega^2+1})^1_1$.
\end{cor}
\begin{proof} By Theorem~\ref{spacefromtree},
it suffices to find a streamlined tree $T\s{}^{<\omega_1}\omega_1$ of size continuum,
and a ladder system $\vec A$ over $T':=T\restriction\acc(\omega_1)$
such that the graph $(T,E_{\vec A})$ is uncountably chromatic.
A tree with the same key features was constructed by D.~Soukup in \cite[Theorem~3.5]{MR3383252}, though it was not streamlined.
By abstract nonsense considerations (see \cite[Lemma~2.5]{paper23}), this should not make any difference. 
As the argument in \cite{paper23} does not deal with the adjacent ladder system, we spell out the details in here.

Soukup's tree is the tree $T(S):=\{ x\s\omega_1\mid \acc^+(x)\s x\s S\}$ for an arbitrary choice of a stationary and co-stationary subset $S$ of $\omega_1$,
ordered by the end-extension relation, $\sq$. It comes equipped with a sequence $\vec C=\langle C_x\mid x\in T(S)\rangle$ such that 
$C_x$ is either a finite subset of $x_\downarrow$ or a cofinal subset of $x_\downarrow$ of order-type $\omega$.
In addition, the corresponding graph $(T(S),\{ \{y,x\}\mid x\in T(S), y\in C_x\})$ is uncountably chromatic.

As $S$ is stationary, $T(S)$ contains infinite sets.
As $S$ is co-stationary, every element of $T(S)$ is countable. Altogether  $|T(S)|=2^{\aleph_0}$.
As every $x\in T(S)$ is a closed countable set of countable ordinals, 
its corresponding collapsing map $\pi_x:\otp(x)\rightarrow x$ is an element of $\bigcup_{\beta\in\nacc(\omega_1)}{}^\beta\omega_1$.
In addition, for every pair $x\sqsubset y$ of nodes in $T(S)$, it is the case that $\pi_x\subset\pi_y$.
Thus, altogether, $$T:=\{ \pi_x\restriction\alpha\mid x\in T(S), \alpha<\omega_1\}$$ is a streamlined tree 
satisfying:
\begin{itemize}
\item $x\mapsto \pi_x$ forms an order-isomorphism from $(T(S),{\sq})$ to $(T\restriction\nacc(\omega_1),{\s})$;
\item every element of $T\restriction\acc(\omega_1)$ admits a unique immediate successor.\footnote{Indeed, the immediate successor of a node $t\in T\restriction\acc(\omega_1)$ is $\pi_x$ for $x:=\im(t)\cup\{\sup(\im(t))\}$.}
\end{itemize}

We shall now define the ladder system $\vec A=\langle A_t\mid t\in T'\rangle$, for $T':=T\restriction\acc(\omega_1)$, as follows.
Given $t\in T\restriction\acc(\omega_1)$, let ${x_t}$ denote the unique element of $T(S)$ such that $\pi_{x_t}$ is the immediate successor of $t$.
Now consider the following possibilities:

$\br$ If $|C_{x_t}|<\aleph_0$, then let $A_t$ be an arbitrary cofinal subset of $t_\downarrow$ of order-type $\omega$. 

$\br$ Otherwise, $C_{x_t}$ is a cofinal subset of $({x_t})_\downarrow$ of order-type $\omega$,
and hence $$A_t:=\{\pi_y\restriction\sup(\otp(y))\mid y\in C_{x_t}\}$$ is a cofinal subset of $t_\downarrow$ of order-type $\omega$.

\begin{claim} The graph $(T,E_{\vec A})$ is uncountably chromatic.
\end{claim}
\begin{proof} Let $c:T\rightarrow\omega$ be given, and we shall find $s\subset t$ such that $c(s)=c(t)$.

As in the proof of Claim~\ref{claim213}, by recursion on the levels of the tree we may construct a coloring $d:T(S)\rightarrow\omega$ satisfying the following for every $x\in T(S)$:
\begin{enumerate}
\item  If $C_x$ is finite, then $d(x)$ is an odd positive integer that does not belong to $\{ d(y)\mid y\in C_x\}$;
\item If $C_x$ is infinite, then $d(x)=c(\pi_x\restriction\sup(\otp(x)))\cdot 2$.
\end{enumerate}

As the graph $(T(S),\{ \{y,x\}\mid x\in T(S), y\in C_x\})$ is uncountably chromatic,
we now pick a pair $y\sqsubset x$ of nodes in $T(S)$ such that $d(y)=d(x)$.
Denote:
\begin{itemize}
\item $t:=\pi_x\restriction\sup(\otp(x))$, and 
\item $s:=\pi_y\restriction\sup(\otp(y))$.
\end{itemize}

As $d(x)=d(y)$, by the choice of $d$, $C_x$ cannot be finite, so the only other option is that
$C_x$ is a cofinal subset of $x_\downarrow$ of order-type $\omega$. In particular, $x_\downarrow$ cannot have a maximal element,
and hence $\otp(x)=\alpha+1$ for some $\alpha\in\acc(\omega_1)$.
Therefore, $\pi_x$ is an immediate successor of the above node $t$, so that $t\in T\restriction\acc(\omega_1)$ and $x_t=x$. 
It thus follows from the definition of $A_t$ that $s\in A_t$.

Finally, as $C_x$ is not finite, $d(x)=c(t)\cdot2$. From $d(y)=d(x)$ being even, we then infer that $d(y)=c(s)\cdot2$.
Altogether, $c(s)=c(t)$, as sought.
\end{proof}
This completes the proof.
\end{proof}

\section{Forcing highly chromatic Hajnal-M\'at\'e graphs}\label{sec2}
A \emph{Hajnal-M\'at\'e graph} is a graph of the form $G=(\kappa,E)$, where $\kappa$ is a cardinal,
$E$ is a subset of $[\kappa]^2$, and for every pair $\beta<\gamma$ of ordinals from $\kappa$,
$\sup\{ \alpha<\beta\mid \{\beta,\gamma\}\in E\}<\beta$.
The existence of an uncountably chromatic Hajnal-M\'at\'e graph over $\omega_1$
gives rise to a tree $T$ and a ladder system $\vec A$ satisfying the hypotheses of Theorem~\ref{spacefromtree} by identifying $\omega_1$ with
the streamlined tree $T:={}^{<\omega_1}1$.

In this section, we highlight a class of notions of forcing that inevitably add highly chromatic Hajnal-M\'at\'e graphs.

\begin{defn} Let $\mathbb P=(P,{\le})$ denote a notion of forcing, and $\lambda$ denote an infinite regular cardinal.
\begin{itemize}
\item $\mathbb P$ is \emph{${}^\lambda\lambda$-bounding} 
iff for every $g\in{}^\lambda\lambda\cap V^{\mathbb P}$,
there exists some $f\in{}^\lambda\lambda\cap V$ such that $g(\alpha)\le f(\alpha)$ for all $\alpha<\lambda$;
\item $\mathbb P$ satisfies the \emph{$\lambda^+$-stationary chain condition}
(\emph{$\lambda^+$-stationary-cc}, for short) iff for every sequence $\langle p_\delta\mid \delta<\lambda^+\rangle$ of conditions in $\mathbb P$
there are a club $D\s\lambda^+$ and a regressive map $h:D\cap E^{\lambda^+}_{\lambda}\rightarrow\lambda^+$ such that for all
$\gamma,\delta\in\dom(h)$, if  $h(\gamma)=h(\delta)$, then $p_\gamma$ and $p_\delta$ are compatible.
\end{itemize}
\end{defn}

\begin{thm} Suppose that $\lambda$ is an infinite regular cardinal,
and $\mathbb P$ is a $\lambda^+$-stationary-cc notion of forcing satisfying at least one of the following:
\begin{enumerate}
\item $\mathbb P$ preserves the regularity of $\lambda$,
and is not ${}^\lambda\lambda$-bounding;
\item $\mathbb P$ forces that $\cf(\lambda)<|\lambda|$. In addition, $\cf(\ns_\lambda,{\s})=\lambda^+$;
\item In $V^{\mathbb P}$, there exists a cofinal subset $\Lambda\s\lambda$ such that for every function $f\in{}^\lambda\lambda\cap V$, there exists some $\xi\in\Lambda$ with $f(\xi)<\min(\Lambda\setminus(\xi+1))$.
\end{enumerate}

Then, in $V^{\mathbb P}$, there exists a sequence $\langle C_\delta\mid \delta\in E^{\lambda^+}_\lambda\rangle$ satisfying the following:
\begin{itemize}
\item For every $\delta\in E^{\lambda^+}_\lambda$, $C_\delta$ is a club in $\delta$ of order-type $\lambda$;
\item For every coloring $c:E^{\lambda^+}_\lambda\rightarrow\lambda$, there are $\gamma,\delta\in E^{\lambda^+}_\lambda$ such that $\gamma\in C_\delta$
and $c(\gamma)=c(\delta)$.
\end{itemize}
\end{thm}
\begin{proof} By \cite[Proposition~3.1]{paper26}, Clause~(3) follows from Clauses (1) and (2), so hereafter, we shall assume Clause~(3).

Work in $V$. Write $\Delta:=E^{\lambda^+}_\lambda$.
For each $\delta\in \Delta$, let $\pi_\delta:\lambda\rightarrow \delta$ denote the inverse collapse of some club in $\delta$,
and let $\psi_\delta:\lambda\leftrightarrow\delta$ be some bijection.

Next, let $G$ be $\mathbb P$-generic over $V$, and work in $V[G]$.
By Clause~(3) and the proof of \cite[Lemma~3.2]{paper26}, we may fix a club $\Lambda\s\lambda$ of order-type $\cf(\lambda)$,
such that for every function $f\in{}^\lambda\lambda\cap V$, $\sup\{\xi\in\Lambda\mid f(\xi)<\min(\Lambda\setminus(\xi+1))\}=\lambda$.

Let $\delta\in \Delta$.
Clearly, $B_\delta:=\pi_\delta[\Lambda]$ is a club in $\delta$ of order-type $\cf(\lambda)$.
Next, let $C_\delta$ be the ordinal closure below $\delta$ of the following set
$$B_\delta\cup\bigcup\{ \psi_\delta[\alpha^+]\cap(\pi_\delta(\alpha),\pi_\delta(\alpha^+))\mid \alpha\in\Lambda\ \&\ \alpha^+=\min(\Lambda\setminus(\alpha+1))\}.$$
Note that, for every pair $\beta<\beta^+$ of successive elements of $\pi_\delta[\Lambda]$, $C_\delta\cap(\beta,\beta^+)$ is 
covered by the closure of $\psi_\delta[\otp(\Lambda\cap\beta^+)]$, which is a set of size $<\lambda$. Therefore, $\otp(C_\delta)\le\lambda$.
\begin{claim}\label{claim411} For every $\Gamma\in[\lambda^+]^{\lambda^+}$ from $V$, for every $\delta\in \acc^+(\Gamma)\cap \Delta$,
it is the case that $\sup(C_\delta\cap \Gamma)=\delta$.
\end{claim}
\begin{proof} Let $\Gamma\in[\lambda^+]^{\lambda^+}$ in $V$. Let $\delta\in\Delta\cap\acc^+(\Gamma)$ and $\epsilon<\delta$; we shall find $\gamma\in \Gamma\cap C_\delta$ above $\epsilon$.
As $\delta\in\acc^+(\Gamma)$, we may define a function $f_0:\lambda\rightarrow\lambda$ via
$$f_0(\alpha):=\min\{\alpha'<\lambda\mid(\pi_\delta(\alpha),\pi_\delta(\alpha'))\cap\Gamma\neq\emptyset\}.$$
Then, we may define a function $f_1:\lambda\rightarrow\lambda$ via:
$$f_1(\alpha):=\min\{ i<\lambda\mid \psi_\delta(i)\in (\pi_\delta(\alpha),\pi_\delta(f_0(\alpha)))\cap\Gamma\}.$$
Define $f:\lambda\rightarrow\lambda$ via $f(\alpha):=\max\{f_0(\alpha),f_1(\alpha)\}$. As $\Gamma\in V$, the function $f$ is in ${}^\lambda\lambda\cap V$,
and hence $A:=\{\xi\in\Lambda\mid f(\xi)<\min(\Lambda\setminus(\xi+1))\}$ is cofinal in $\lambda$.
Pick a large enough $\alpha\in A$ such that $\pi_\delta(\alpha)\ge\epsilon$. Denote $\alpha^+:=\min(\Lambda\setminus(\alpha+1))$.
Then $\alpha':=f_0(\alpha)$ and $i:=f_1(\alpha)$ are both less than $<\alpha^+$.
So 
$$\psi_\delta(i)\in\psi_\delta[\alpha^+]\cap (\pi_\delta(\alpha),\pi_\delta(\alpha^+))\cap\Gamma,$$
meaning that $\psi_\delta(i)$ is an element of $C_\delta\cap\Gamma$ above $\epsilon$.
\end{proof}

Work in $V$.
Suppose that $p$ is a condition forcing that $\dot c$ is  a name for a function from $\Delta$ to $\lambda$.
For each $\delta\in\Delta$, let $p_\delta$ be a condition extending $p$ and deciding $\dot c(\delta)$ to be, say, $\tau_\delta$.
Fix a club $D\s\lambda^+$ and a regressive map $h:D\cap E^{\lambda^+}_{\lambda}\rightarrow\lambda^+$ such that for all
$\gamma,\delta\in\dom(h)$, if  $h(\gamma)=h(\delta)$ then $p_\gamma$ and $p_\delta$ are compatible.

Find $(\tau,\eta)\in\lambda\times\lambda^+$ for which $$\Gamma:=\{\delta\in\Delta\cap D\mid \tau_\delta=\tau\ \&\ h(\delta)=\eta\}$$ is stationary.
As $\acc^+(\Gamma)$ is a club (in $V$), Claim~\ref{claim411} provides us with a $\delta\in\Gamma$ such that $\sup(C_\delta\cap\Gamma)=\delta$.
Pick $\gamma\in C_\delta\cap\Gamma$. As $h(\delta)=\eta=h(\gamma)$, we may pick some $q$ extending $p_\delta$ and $p_\gamma$. 
Then, $q$ is an extension of $p$ forcing that $\gamma,\delta\in\Delta$ and $c(\gamma)=\tau=c(\delta)$.
\end{proof}

\begin{cor} If $\lambda$ is a measurable cardinal, then in the forcing extension by Prikry forcing using a normal measure on $\lambda$, 
there exists a Hajnal-M\'at\'e graph over $\lambda^+$ of chromatic number $\lambda^+$.\qed
\end{cor}

\begin{cor} After forcing to add any number of Cohen reals, there is an uncountably chromatic Hajnal-M\'at\'e graph over $\omega_1$.\qed
\end{cor}

Putting the preceding together with Theorem~\ref{spacefromtree}, we obtain Theorem~\ref{thmb}:

\begin{cor} After forcing to add any number of Cohen reals, there exists a zero-dimensional regular space $X$ of size $\aleph_1$, of character $\mathfrak b$,
satisfying $X\rightarrow(\top{\omega+1})^1_\omega$, but not $X\rightarrow(\top{\omega^2+1})^1_1$.\qed
\end{cor}

\section{Acknowledgments}

The first author was supported by the European Research Council (grant agreement ERC-2018-StG 802756).
The second author was partially supported by the Israel Science Foundation (grant agreement 203/22)
and by the European Research Council (grant agreement ERC-2018-StG 802756).

Some of the results of this paper were presented by the first author at 
the \emph{Winter School in Abstract Analysis} meeting in \v{S}t\v{e}ke\v{n},  Czech Republic, January 2023.
He thanks the organizers for the opportunity to speak and the participants for their feedback.


\begin{thebibliography}{CFJ23}

\bibitem[BR19]{paper26}
Ari~Meir Brodsky and Assaf Rinot.
\newblock More notions of forcing add a Souslin tree.
\newblock {\em Notre Dame J. Form. Log.}, 60(3):437--455, 2019.

\bibitem[BR21]{paper23}
Ari~Meir Brodsky and Assaf Rinot.
\newblock A microscopic approach to Souslin-tree constructions. Part~II.
\newblock {\em Ann. Pure Appl. Logic}, 172(5):Paper No. 102904, 65, 2021.

\bibitem[CFJ23]{MR4550370}
Rodrigo Carvalho, Gabriel Fernandes, and L\'{u}cia~R. Junqueira.
\newblock Partitions of topological spaces and a new club-like principle.
\newblock {\em Proc. Amer. Math. Soc.}, 151(4):1787--1800, 2023.

\bibitem[KW87]{MR911048}
P{\'e}ter Komj\'{a}th and William Weiss.
\newblock Partitioning topological spaces into countably many pieces.
\newblock {\em Proc. Amer. Math. Soc.}, 101(4):767--770, 1987.

\bibitem[Rin22]{paper51}
Assaf Rinot.
\newblock On the ideal $J[\kappa]$.
\newblock {\em Ann. Pure Appl. Logic}, 173(2):Paper No. 103055, 13pp, 2022.

\bibitem[Sou15]{MR3383252}
D\'{a}niel~T. Soukup.
\newblock Trees, ladders and graphs.
\newblock {\em J. Combin. Theory Ser. B}, 115:96--116, 2015.

\end{thebibliography}
\end{document}